\documentclass{amsart}

\usepackage{graphicx}
\usepackage{mathrsfs, amsfonts}
\usepackage[all]{xy}
\usepackage{verbatim}
\usepackage{hyperref}

\newtheorem{theorem}{Theorem}[section]
\newtheorem{crl}[theorem]{Corollary}
\newtheorem{lemma}[theorem]{Lemma}
\newtheorem{proposition}[theorem]{Proposition}

\theoremstyle{definition}
\newtheorem{definition}[theorem]{Definition}
\newtheorem{example}[theorem]{Example}

\theoremstyle{remark}
\newtheorem{remark}[theorem]{Remark}

\numberwithin{equation}{section}

\newcommand{\df}{\emph}

\usepackage{color}
\usepackage[dvipsnames]{xcolor}

\newcommand{\F}{\mathcal{F}}

\newcommand{\SM}{(\mathbf{S}, \mathbf{M})}
\newcommand{\BSM}{\mathbf{B\SM}}
\newcommand{\x}{\mathbf{x}}
\newcommand{\X}{\mathscr{X}}
\newcommand{\y}{\mathbf{y}}
\newcommand{\Z}{\mathbb{Z}}
\newcommand{\ZP}{\Z\mathbb{P}}

\begin{document}


\title{Number of Triangulations of a M\"obius Strip}

\author{Véronique Bazier-Matte}
\address{Department of Mathematics, University of Connecticut, Storrs, Connecticut 06269-1009, United States of America}
\email{veronique.bazier-matte@uconn.edu}

\author{Ruiyan Huang}
\address{Lester B. Pearson College, BC, Canada}
\email{rhuang46@pearsoncollege.ca}

\author{Hanyi Luo}
\address{The Affiliated High School of South China Normal University}
\email{luohy.laurie2017@gdhfi.com}






\begin{abstract}
Consider a M\"obius strip with $n$ chosen points on its edge. A triangulation is a maximal collection of arcs among these points and cuts the strip into triangles. In this paper, we proved the number of all triangulations that one can obtain from a M\"obius strip with $n$ chosen points on its edge is given by $4^{n-1}+\binom{2n-2}{n-1}$, then we made the connection with the number of clusters in the quasi-cluster algebra arising from the M\"obius strip.
\end{abstract}

\maketitle
\tableofcontents

\section{Introduction}


\renewcommand{\S}{\mathbf{S}}

We consider the M\"obius strip as a \textit{marked surface}, that is, a two-dimensional Riemann surface $\mathbf{S}$ with a set $\mathbf{M}$ of marked points on the boundary of $\S$. We cut this surface into triangles with \emph{arcs}, which are curves on the marked surface. A maximal collections of arcs without intersections is called a \textit{triangulation}. Such an example of triangulation is given at Figure \ref{fig:mobiustrig}.

\begin{figure}[h]\centering\includegraphics[width=0.3\textwidth]{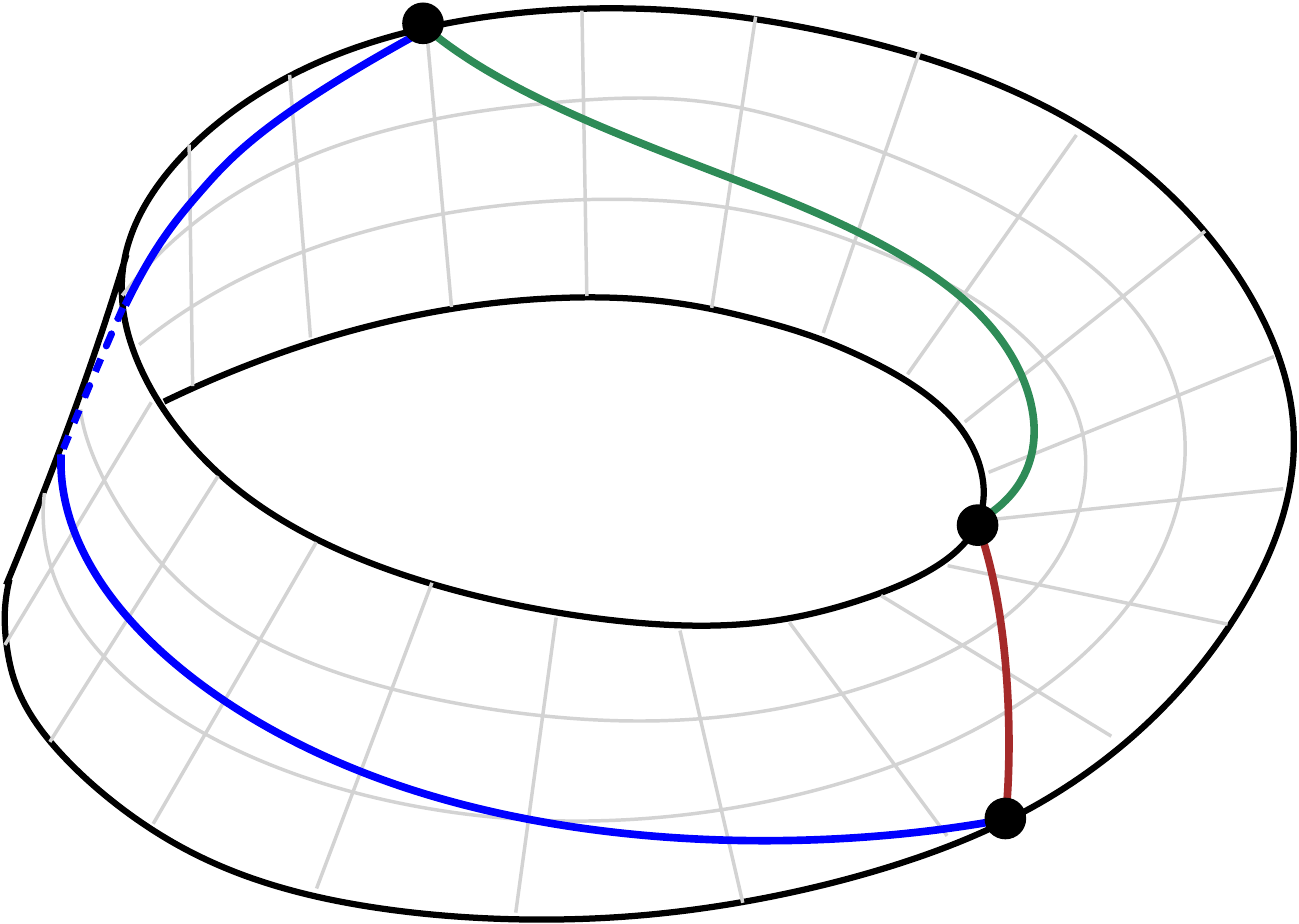}\caption{A triangulation of M\"obius strip with three marked points}\label{fig:mobiustrig}\end{figure}

We initiate a study about the number of triangulations of the M\"obius strip, the only non-oriented surface with a finite number of triangulations.

In the theory of quasi-cluster algebras developed by Dupont and Palesi \cite{DP15}, the triangulations of a M\"obius strip 
correspond naturally to the clusters in the quasi-cluster algebra from a M\"obius strip, which is explained in section \ref{Sect:QuasiCluster}.

Our main result is the following theorem:
\begin{theorem}
    The number of triangulations of a M\"obius strip with $n$ marked points is given by
        $$4^{n-1}+\binom{2n-2}{n-1}.$$
\end{theorem}

To prove this result, we count the number of triangulations in an iterative way. Then we simplify the equation into a more succinct form.

\section{Preliminaries}

In this section, we start by reviewing some notions about non-orientable surfaces, and aim to define the triangulations of these surfaces. The review combines results from \cite{FST08} and \cite{DP15}.

Let $\mathbf{S}$ be a compact 2-dimensional manifold with boundary and let $\mathbf{M}$ be a finite set of marked points of $\mathbf{S}$ such that each boundary component contains at least one marked point. We denote the surface $(\mathbf{S}, \mathbf{M})$. In the scope of this paper, we only consider the case when all the points of $\mathbf{M}$ are on the boundary. Moreover, to exclude pathological cases, we do not allow $\mathbf{S}$ to be a monogon or digon.

Intuitively, a surface is \df{orientable} if moving continuously any figure on this surface cannot result in the mirror image of the figure when it is back to its starting point. Else, it is \emph{non-orientable}.

\begin{definition}
Let $\mathbf{S}$ be a surface. A closed curve on $\mathbf{S}$ is defined as \df{two-sided} if it has a regular, orientable neighbourhood. If any neighbourhood is non-orientable, the curve is \df{one-sided}.
\end{definition}

A surface is non-orientable if and only if there is a one-sided curve on it. A M\"obius strip is such an example. This will be an important notion in the proof of Proposition \ref{Prp::Tn}. 

\begin{definition}
A \emph{cross-cap} is obtained from an annulus by identifying antipodal points on the inner boundary.
This is represented (as in Figure \ref{fig:crosscap}) by drawing a cross inside the inner circle.
\end{definition}

\begin{figure}[h]\centering\includegraphics[width=0.3\textwidth]{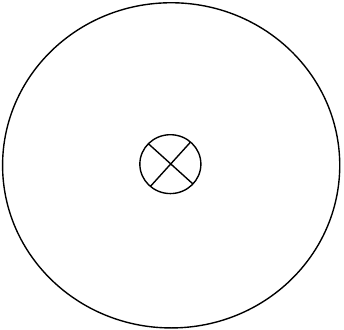}\caption{Representation of a M\"obius strip with cross-cap}\label{fig:crosscap}\end{figure}


\begin{definition}
The \df{isotopy class} of a curve $\gamma$ is the set of embeddings of $\gamma$ that can be deformed into each other through a continuous path of homomorphisms such that:
\begin{itemize}
\item the endpoints of each embedding are always the same;
\item except for its endpoints, each embedding is disjoint from $\mathbf{M}$ and the boundary.
\end{itemize}
We define two types of curves without self-intersections in $\SM$. The first type of curves without self-intersections consists of closed one-sided curves in the interior of $S$. The second type consists of curves with both endpoints in $M$. We refer to the union of isotopy classes of these two types as \emph{arcs}.

Two arcs are \df{compatible} if there exist representatives in their isotopy class not intersecting with each other.
\end{definition}

\begin{figure}\centering\includegraphics[width=0.7\textwidth]{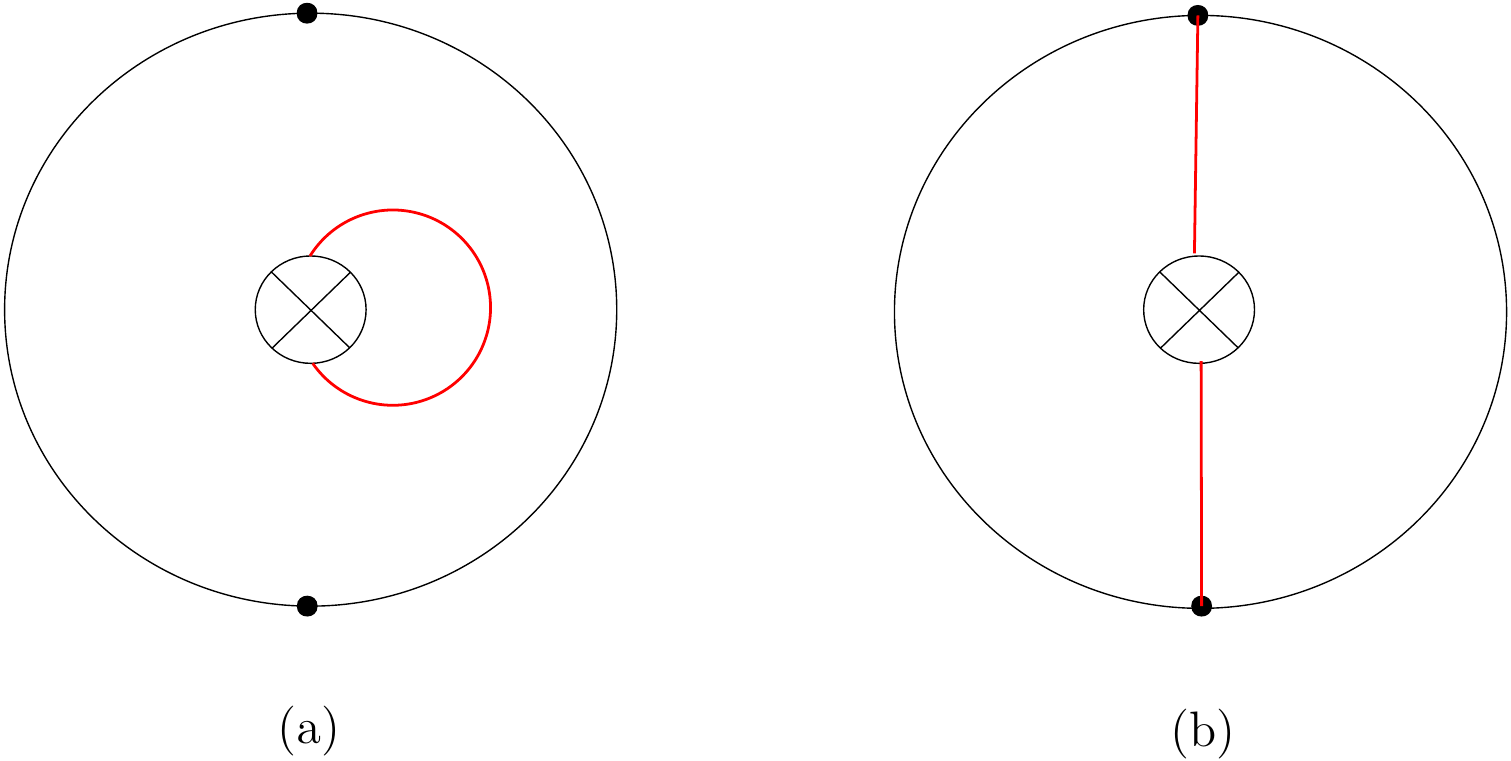}
\caption{The two kinds of arcs}
\label{fig:arc}\end{figure}

\begin{example}
In figure \ref{fig:compatible}, even if they intersect each other, the curves $\alpha$ and $\beta$ in (a) are representatives of compatible arcs, because $\alpha'$ and $\beta'$ in (b) are representatives of the same respective isotopy classes and they do not intersect.

\begin{figure}\centering\includegraphics[width=0.8\textwidth]{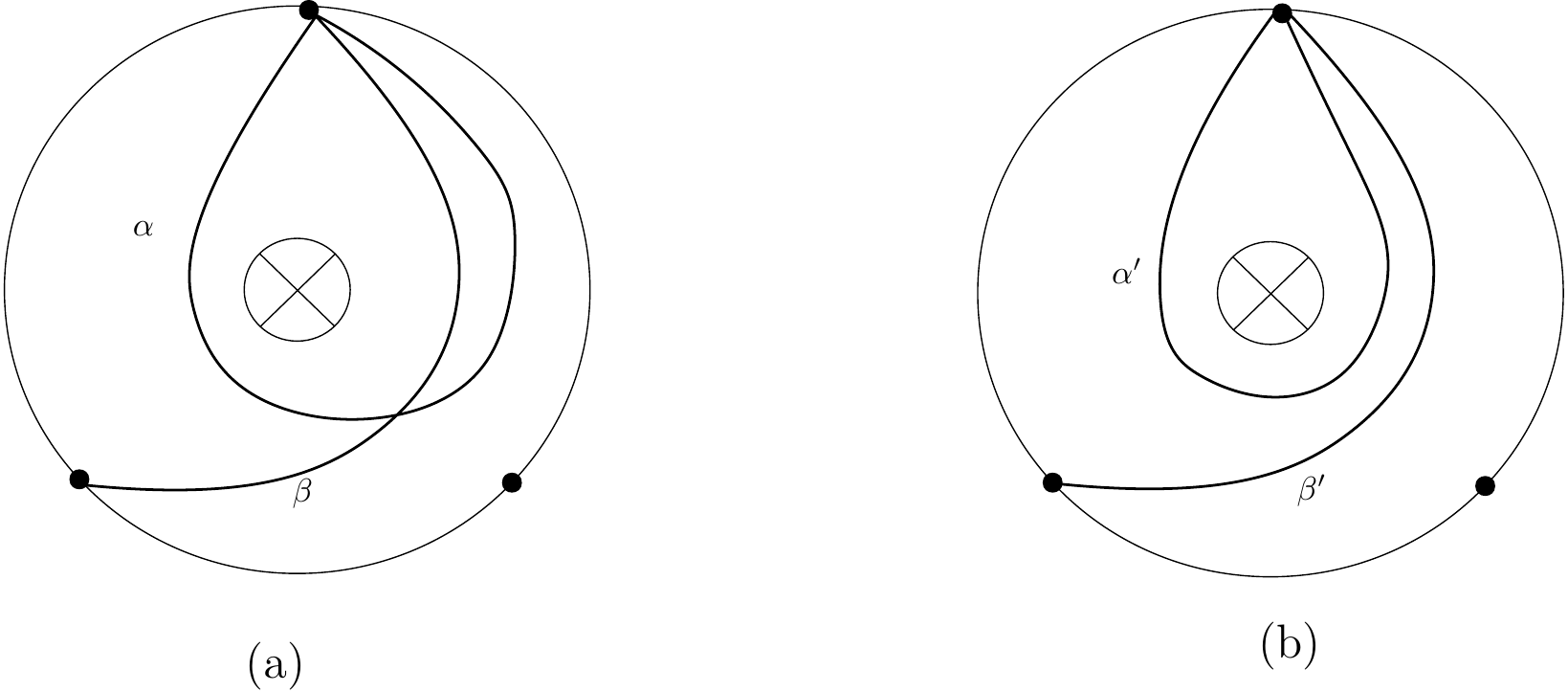}\caption{Compatible arcs}\label{fig:compatible}\end{figure}
\end{example}

Note that what we refer to as arcs, Dupont and Palesi refer to as quasi-arcs, \cite[Definition 1]{DP15}.



Let $\BSM$ denote the set of \emph{boundary segments}, that is, the set of connected components of the boundary of $\mathbf{S}$ with the points of $\mathbf{M}$ removed.


With the above definitions, we are now able to introduce triangulations of a surface.

\begin{definition}
A \df{triangulation} of $(\mathbf{S}, \mathbf{M})$ is a maximal collection of compatible arcs. 
Each area of $\SM$ cut by this maximal collection of compatible arcs is a \emph{triangle}.
\end{definition}

\begin{example}
The Figure \ref{fig:cros-trig} is the same triangulation as in Figure \ref{fig:mobiustrig}.

\begin{figure}[h]\centering\includegraphics[width=0.3\textwidth]{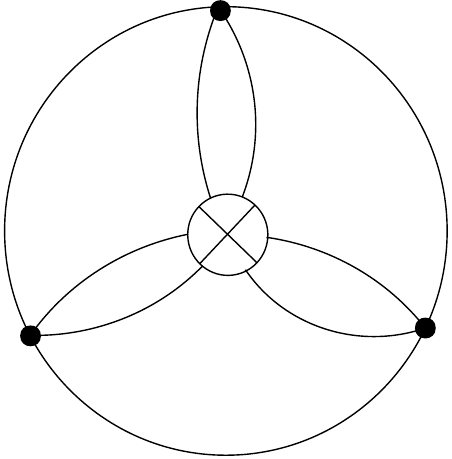}\caption{A triangulation of M\"obius strip with cross-cap representation}\label{fig:cros-trig}\end{figure}
\end{example}

Remark that the definition implies that if another arc can be added into the current collection of arcs, then this is not a triangulation.

Once again, our terminology does not totally coincide with that of Dupont and Palesi. What we call a triangulation, they would call a quasi-triangulation, \cite[Defintion 6]{DP15}.

Let $M_n$ denote the M\"obius strip with $n$ marked points on its boundary.

\begin{example}
Figure \ref{fig:quasi-trig} shows a triangulation of $M_2$.
\end{example}

\begin{remark}
In some cases, the triangles formed as a result of triangulation can have two edges identified; these triangles are called \df{anti-self-folded triangles}.
The identified edges must 
go through the cross-cap and all the three vertices are also identified.
\end{remark}

\begin{example}
 There is an anti-self-folded triangle in the Figure \ref{fig:quasi-trig} formed by the arcs $a$ and $b$.
 The three edges are represented with different colors, with the yellow one and the green one identified.  
\end{example}

\begin{figure}[hb]\centering\includegraphics[width=0.3\textwidth]{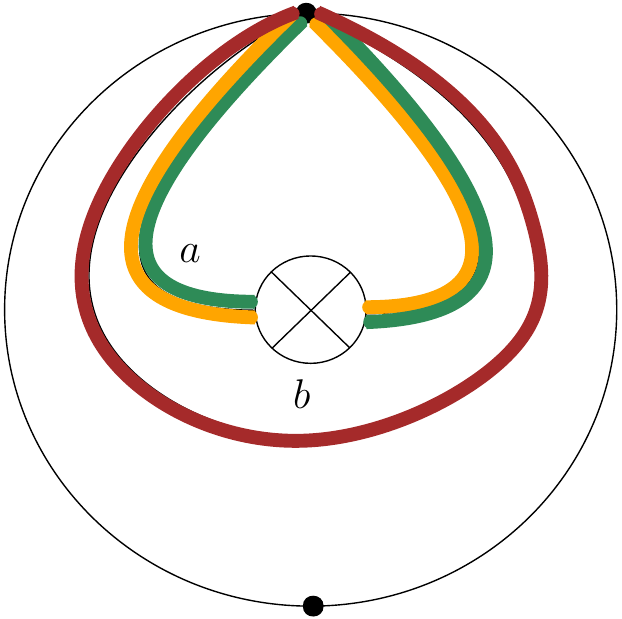}\caption{Triangulation of the Möbius strip with two marked points, with the three edges of an anti-self-folded triangle labeled}\label{fig:quasi-trig}\end{figure}

\begin{definition}
A \emph{quasi-triangle} consists in a monogon encircling only a cross-cap and a one-sided curve through the cross-cap in the interior of $\mathbf{S}$, as in Figure \ref{fig:quasitrig}.
\end{definition}

\begin{figure}[h]\centering\includegraphics{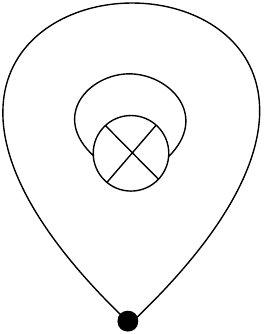}\caption{A quasi-triangle}\label{fig:quasitrig}\end{figure}

\begin{proposition} \cite[Proposition 7]{DP15}
 A triangulation cuts $M_n$ into a finite union of triangles and quasi-triangles.
\end{proposition}


\begin{proposition} \cite[Example 13]{DP15}
The number of arcs in a triangulation of $M_n$ is equal to $n$ the number of marked points.
\end{proposition}

\begin{theorem}\cite[Theorem 45]{DP15} \label{Thm::NumberArcs}
Let $\SM$ be a surface with finite number of triangulations and marked points on the boundary of $\S$. If $\SM$ is orientable, it is an $n$-gon ($n \geq 3$); if $\SM$ is non-orientable, it is a M\"obius strip.
\end{theorem}


\section{Number of Triangulations of A M\"obius Strip}

Recall that the Catalan numbers are defined recursively as: 
\begin{equation} \label{eq:Cn1}
    C_{n} = \sum_{i=0}^{n-1} C_iC_{n-1-i}
= \frac{(4n-2)C_{n-1}}{n+1},
\text{ for } n \geq 1, \; C_0 = 1 
\end{equation}

and is defined explicitly by
\begin{equation} \label{Eq::Cn}
C_n = \frac{1}{n+1} \binom{2n}{n}, 
\end{equation}

\begin{lemma} \label{Lmm::ngone}
The number of triangulations of a $(n+2)$-gon is given by $C_{n}$, \cite{Lam38}.
\end{lemma}


We provide this simple proof because we will use a similar approach to the M\"obius strip case.

\begin{proof}
We prove by induction that $C_{n}$ is the number of triangulations of a $(n+2)$-gon.

Label the vertices of the $(n+2)$-gon from $1$ to $n+2$.
For each triangulation of the $(n+2)$-gon, there is a single triangle containing the edge between the vertices $1$ and $n+2$. Denote  the third vertex of this triangle by $i$. By the induction hypothesis, there are $C_{i-2}$ triangulations of the $i$-gon on one side of this triangle and $C_{n+1-i}$ triangulations of the $(C_{n+3-i})$-gon on the other side of this triangle, see Figure \ref{fig:n_gon}. Here $2 \leq i \leq n+1$. Therefore, the number of triangulations of a $(n+2)$-gon is 
\[\sum_{i=2}^{n+1}C_{i-2}C_{n+1-i}=\sum_{i=0}^{n-1} C_iC_{n-1-i} = C_{n}.\]
\end{proof}

\begin{figure}[h]\centering\includegraphics[width=0.3\textwidth]{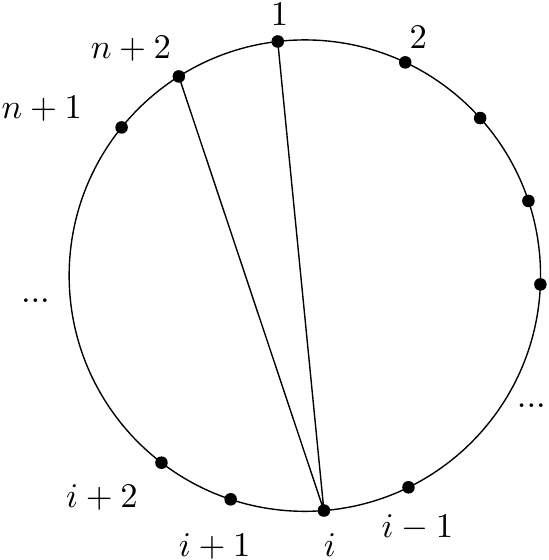}\caption{Number of triangulations in a $(n+3)$-gon}\label{fig:n_gon}\end{figure}

\begin{example}
The pentagon has 5 triangulations shown in the Figure \ref{fig:pentagone}.
\begin{figure}\centering\includegraphics[width=0.6\textwidth]{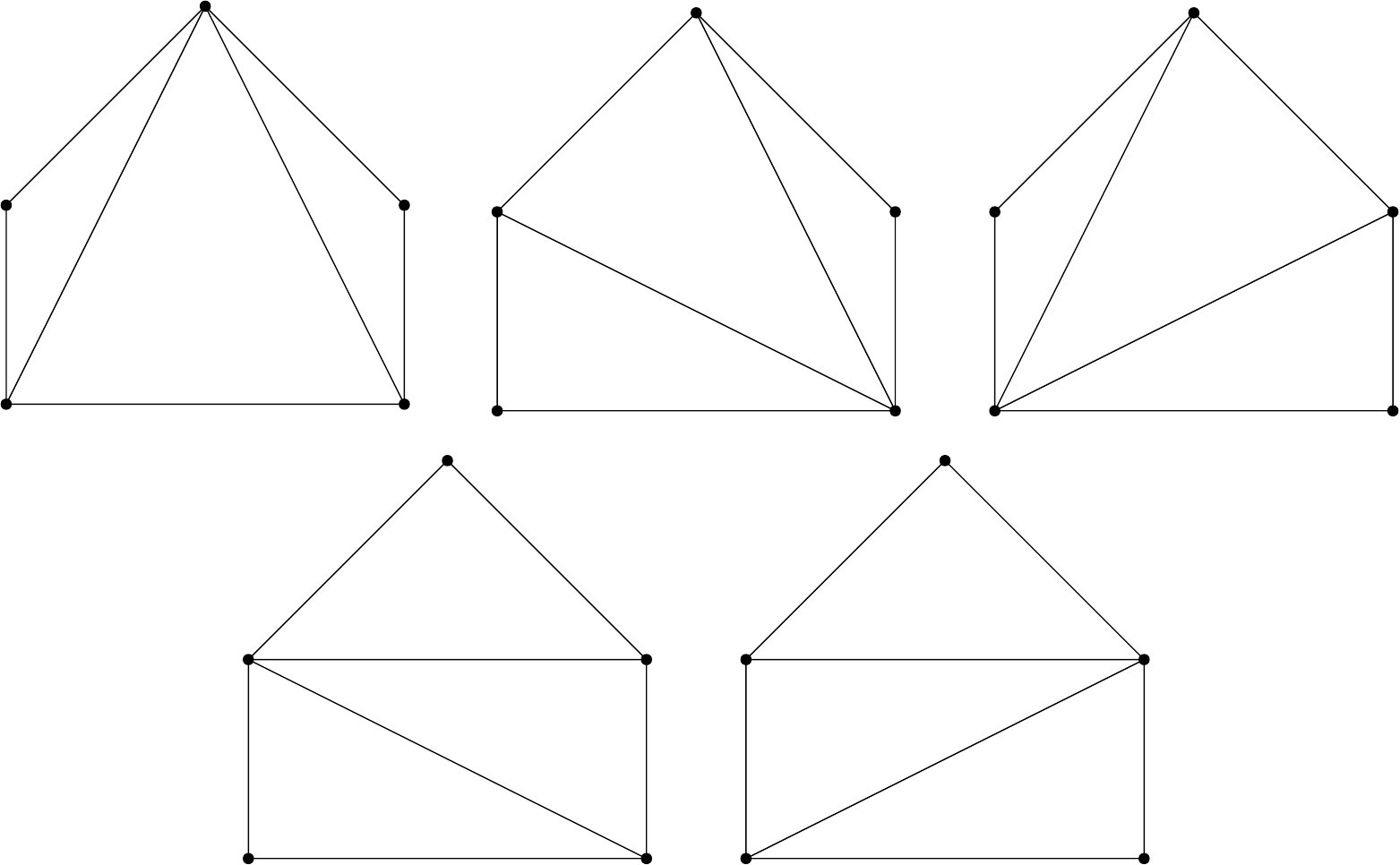}\caption{Triangulations of a pentagon}
\label{fig:pentagone}\end{figure}
\end{example}

\begin{proposition} \label{Prp::Tn}
The number of triangulations of a M\"obius strip with $n$ marked points is given by $T_n$, where 
$$T_n = 2 \sum_{i=0}^{n-2}C_iT_{n-i-1} + nC_{n-1} \text{ for } n\geq2, \text{ with }  T_1 = 2.$$
\end{proposition}

\begin{proof}
Let $(\mathbf{S}, \mathbf{M})$ be a marked surface and let $\mathbf{T}$ be a triangulation of $(\mathbf{S}, \mathbf{M})$.

Recall that in any triangulation, each boundary edge is part of one and only one triangle. 
Therefore, there is a single triangle containing the edge between the marked points $n$ and $1$; denote this triangle $\Delta$. It must be of one of the following three forms shown in Figure \ref{fig:Delta}.


\begin{figure}[htbp]\centering\includegraphics[width=1\textwidth]{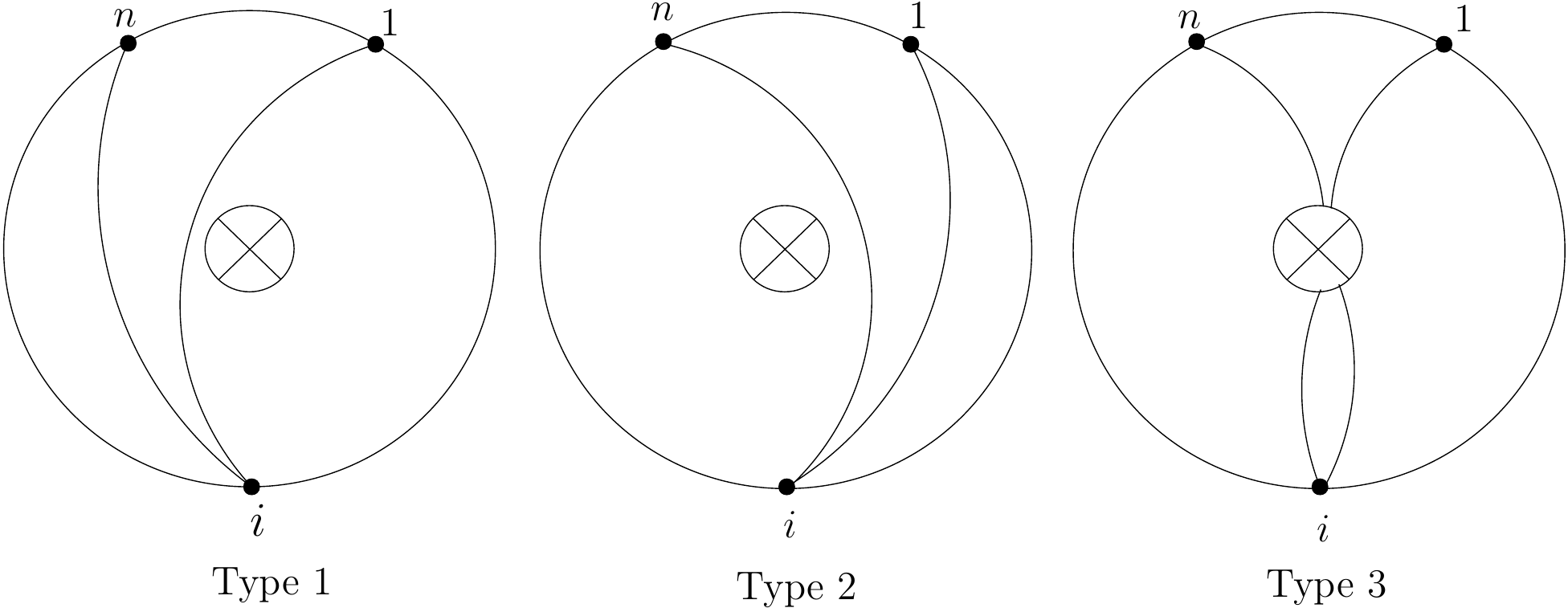}\caption{The three types of triangle containing the edge between the vertex $n$ and $1$} \label{fig:Delta}\end{figure}


Indeed, a triangle cannot contain a cross-cap: a surface with one boundary component, three marked points, and a cross-cap is a M\"obius strip with three marked points instead of a triangle. 

Moreover, a triangle cannot have an edge going through a cross-cap. On the contrary, suppose that a triangle $\Delta$ has such an edge. Then, since we can add another arc in $\Delta$, as depicted by the dotted arc in Figure \ref{Fig::TriangleCrosscap}, $T$ was not part of maximal collection of compatible arcs, i.e a triangulation, so $T$ was not a triangle.

\begin{figure}[htbp]\centering\includegraphics[width=0.3\textwidth]{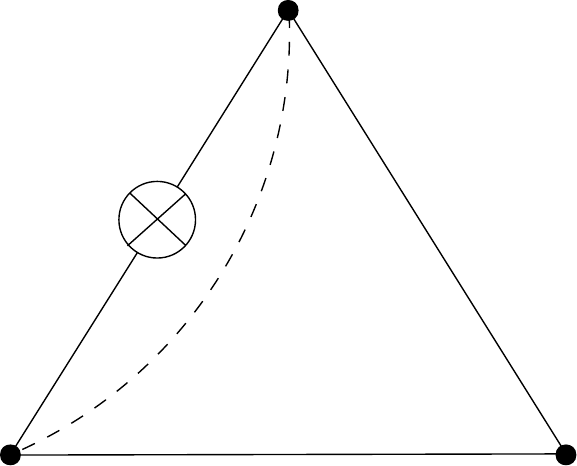}\caption{Triangle $\Delta$ with one side going through the cross-cap} \label{Fig::TriangleCrosscap}\end{figure}

Therefore, we proceed with our proof in three cases. Since Type 1 and Type 2 are mirror images, the number of triangulations occurring in each type are the same, and we discuss them in one case.

\textbf{Case 1:} Consider the triangle $\Delta$ of the form Type 1 or Type 2. Without loss of generality, we consider the case of Type 2. 

We count the number of triangulations by considering the three areas cut out by the two arcs, see Figure \ref{fig:Area}.

\begin{figure}[htbp]\centering\includegraphics[width=1\textwidth]{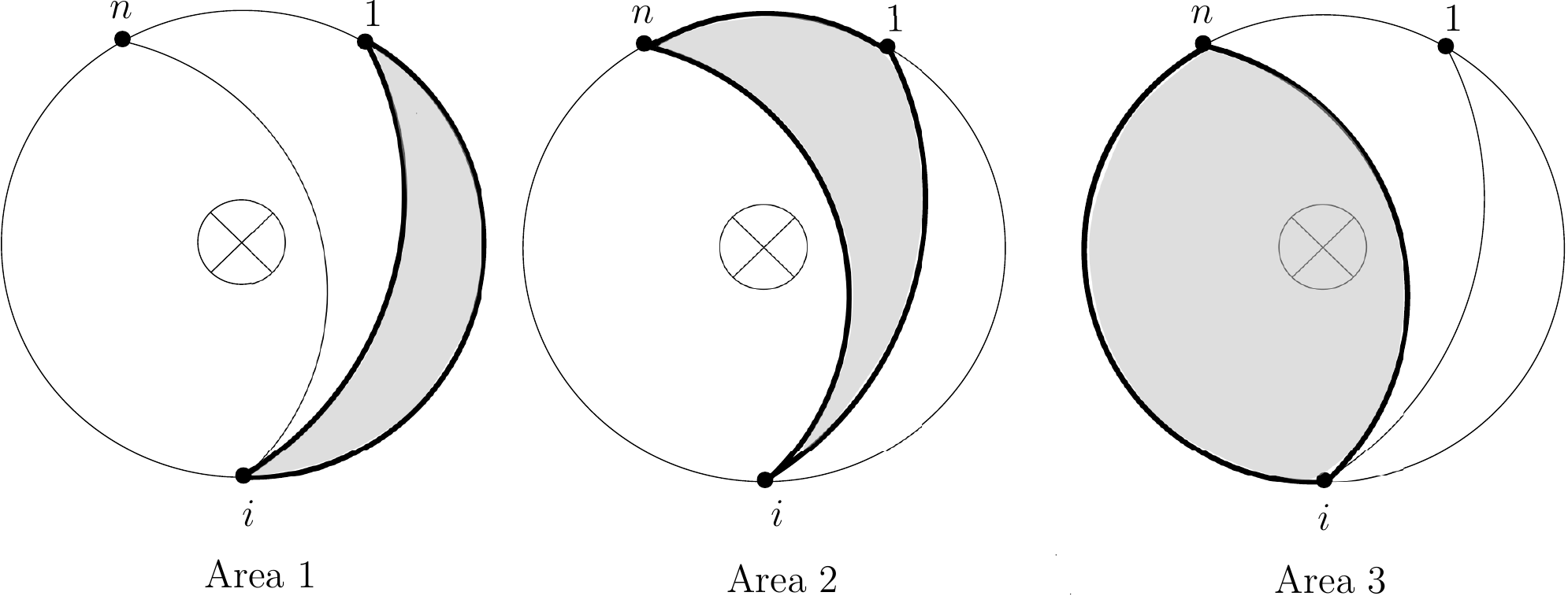}\caption{Areas in Type 2 triangulation} \label{fig:Area}\end{figure}

Area 1 includes $i$ marked points numbered from 1 to $i$. 
Since Area 1 is orientable, 
it is isotopic to an $i$-gon. 
Thus by the Lemma \ref{Lmm::ngone}, Area 1 has $C_{i-2}$ triangulations. 

There is no more compatible arc in Area 2; alternatively, it is a triangle. Therefore, there is only one triangulation in this area.

Area 3 is a M\"obius strip with $n-i+1$ marked points. Hence, there are $T_{n-i+1}$ triangulations in Area 3. 

Thus, the number of triangulations in this case is the product of triangulations coming from these three areas. The number of triangulations containing a triangle $\Delta$ of Type 2
is 
$$\sum_{i=2}^nT_{n-1+1}\times1\times C_{i-2}=\sum_{i=0}^{n-2}T_{n-i-1} C_{i}.$$

Thus, the total number of triangulations containing triangle $\Delta$ of Type 1 and 2 in Case 1 is  
$$ 2\sum_{i=0}^{n-2}T_{n-i-1} C_{i}.$$

\textbf{Case 2:} 
In this case, we consider $\Delta$ of Type 3, in which both arcs go through the cross-cap. The M\"obius Strip is thus divided into two areas, see Figure \ref{fig:areas_3}.

\begin{figure}[htbp]\centering\includegraphics[width=0.8\textwidth]{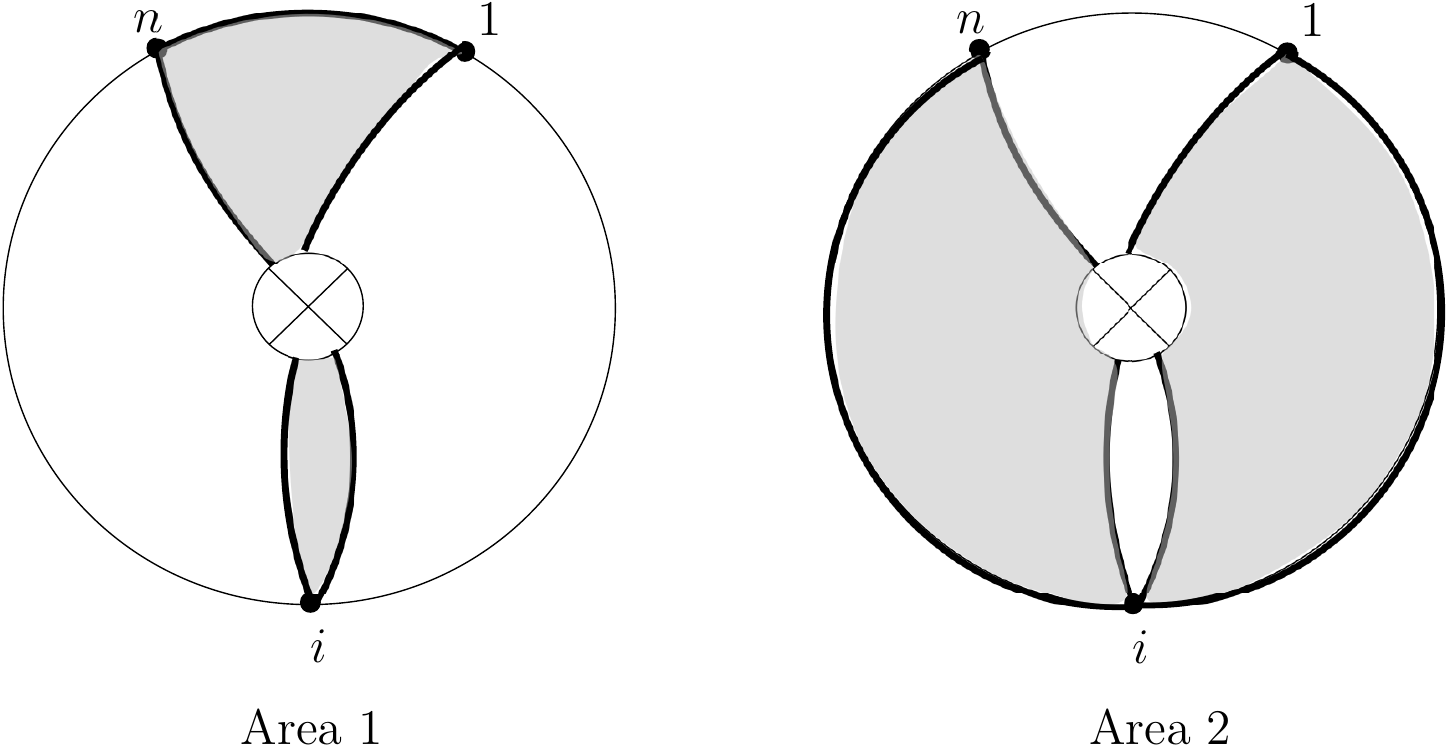}\caption{Areas in Type 3 triangulation} \label{fig:areas_3}\end{figure}

Area 1 is a triangle, and is thus already a triangulation and no other arcs can be added.

Consider the number of triangulations in Area 2. 

We first notice that the Area 2 is isotopic to an $(n+1)$-gon, because: 
\begin{itemize}
    \item $i-1$ edges between vertices $1$ and $i$,
    \item one edge between vertices $i$ and $n$ going through the cross-cap,
    \item $n-i$ edges between vertices $n$ and $i$,
    \item one edge between vertices $i$ and $1$ going through the cross-cap.
\end{itemize}
This gives us a total of $n+1$ edges; since the interior of Area 2 contains no cross-cap, it is an $(n+1)$-gon.
Therefore, it has $C_{n-1}$ triangulations.



Since in this case, $1 \leq i \leq n$, the total number of triangulations can be found by 
\begin{equation}\label{Eq:case2}
    \sum_{i=1}^{n}C_{n-1} = nC_{n-1}
\end{equation}

Therefore, taking all three cases into consideration, the total number of triangulations $T_n$ on a M\"obius strip with $n$ marked points is the sum: 
\begin{align*}
T_n = 2 \sum_{i=0}^{n-2} C_iT_{n-i-1} + nC_{n-1} \text{ for } n\geq2
,\; T_1 = 2
\end{align*}
\end{proof}

We now want to simplify this formula and obtain a direct formula for the number of triangulations of a M\"obius strip.

\begin{proposition}
The number of triangulations of a M\"obius strip with $n$ marked points is $$T_n = 4^{n-1}+ \binom{2n-2}{n-1}.$$
\end{proposition}

\begin{proof}
We proceed by induction to show that
\begin{equation} \label{Eq1}
2 \sum_{i=0}^{n-2} C_iT_{n-i-1} = 4^{n-1}.    
\end{equation}
Clearly, when $n=2$, we have
\[ 2 \sum_{i=0}^{n-2} C_iT_{n-i-1} = 2 \sum_{i=0}^{0} C_iT_{i+1} = 2 (C_0 T_1) = 4 = 4^{1} = 4^{n-1}, \]
which verifies Equation \ref{Eq1}.
Now, suppose Equation \ref{Eq1} is true for all $k < n$ for some positive integer $n$ and prove that it is also verified for $k=n$.
Using both our induction hypothesis and Proposition \ref{Prp::Tn}, for $k < n$, we know that
\[ T_k = 2 \sum_{i=0}^{k-2}C_iT_{k-i-1} + kC_{k-1} = 4^{k-1} + kC_{k-1}. \]
By Equation \ref{eq:Cn1}, we obtain, for $2 \leq k < n$,
\[kC_{k-1} = (4k-6)C_{k-2}.\]
Using this equation, we compute:
\begin{align*}
    T_k &= 4^{k-1} + kC_{k-1}\\
        &= 4\left(4^{k-2} + (k-1)C_{k-2}\right) - 4(k-1)C_{k-2} + kC_{k-1}\\
        &= 4T_{k-1} - 4(k-1)C_{k-2} + (4k-6)C_{k-2}\\
        &= 4T_{k-1} - 2C_{k-2}.
\end{align*}
Therefore,
\begin{align*}
    2 \sum_{i=0}^{n-2} C_i T_{n-1-i} 
        &= 2 \sum_{i=0}^{n-3} C_i T_{n-1-i} + 2 C_{n-2}T_1 \\
        &= 2 \sum_{i=0}^{n-3} C_i \left( 4T_{n-i-2} -2C_{n-i-3} \right) + 4C_{n-2} \\
        &= 4\left( 2\sum_{i=0}^{n-3} C_i T_{n-1-i} \right) - 4 \sum_{i=0}^{n-3} C_i C_{n-i-3} + 4C_{n-2}\\
        &= 4 \cdot 4^{n-2} -4C_{n-2} + 4C_{n-2} \\
        &= 4^{n-1}.
\end{align*}
Since we have now proven equation \ref{Eq1}, we have now shown that
\[ T_{n}= 4^{n-1} + nC_{n-1}. \]

Let's simplify this formula again: by Equation \ref{Eq::Cn}, we have
\[nC_{n-1} = \frac{n}{n+1-1}\binom{2n-2}{n-1} = \binom{2n-2}{n-1}.\]
Hence, by Proposition 3.3, 
\[ T_n=4^{n-1} + \binom{2n-2}{n-1}. \]
\end{proof}

\section{Connection to Quasi-Cluster Algebras} \label{Sect:QuasiCluster}
In this section, we apply our previous results to quasi-cluster algebras, which are generalizations of cluster algebras.  

Introduced in 2002 by Fomin and Zelevinsky \cite{FZ02, FZ03, FZ07}, cluster algebras create an algebraic framework for canonical bases and total positivity. They related to various areas of mathematics such as combinatorics, representation theory of algebras \cite{BMRT07}, 
mathematical physics \cite{AHBHY18}, Poisson geometry \cite{GSV10}, Lie theory \cite{GLS13}, and, of course, triangulation of surfaces \cite{FST08}, just to name a few.
Indeed, in 2008, Fomin, Shapiro and Thurston introduced a particular class of cluster algebras, called \textit{cluster algebras from surfaces}, \cite{FST08}. Later work from Dupont and Palesi introduced the \textit{quasi-cluster algebras} from non-orientable surfaces \cite{DP15}. Then, Wilson studied the shellability and sphericity of the arc complex \cite{Wil18}, and proved the positivity theorem for quasi-cluster algebras \cite{Wil19}.  


In this section, we will not define cluster algebras, but rather quasi-cluster algebras, which generalize them.


Quasi-cluster algebras are commutative $\mathbb{Z}$-algebras of Laurent polynomials with a distinguished set of generators, called \emph{cluster variables}, grouped into sets called \emph{clusters}. The set of all cluster variables is constructed recursively from a set of initial cluster variables using an involutive operation called \emph{mutation}. Let's define these concepts formally.

The difference between classical cluster algebras from surfaces, as defined by Fomin and Zelevinsky \cite{FZ02} and Fomin, Shapiro Thurston, \cite{FST08} and quasi-cluster algebras, as defined by Dupont and Palesi, \cite{DP15}, is that cluster algebras arise only from orientable surfaces, while quasi-cluster algebras arise from both orientable and non-orientable surfaces.

Let $\F$ denote the field of rational functions in $n+k$ indeterminates, that is
\[ \F = \left\{ \left. \frac{p(a_1, \dots, a_{n}, b_1, \dots b_k)}{q(a_1, \dots, a_{n}, b_1, \dots, b_k)} \right| p, q \in \Z[a_1, \dots, a_{n}, b_1, \dots, b_k]\right\}. \]

\begin{definition}
Let $K$ be a field and $k$ a sub-field of $K$. A set $X = \left\{ x_1, \dots, x_m \right\} \subseteq K$ is \df{algebraically independent} on $k$ if $(x_1, \dots x_m)$ is not a root for any non-zero polynomial on $k \left[ t_1, \dots, t_m \right]$.
\end{definition}

\begin{example}
The singletons $\{ \pi \}$ and $\{ \sqrt{2\pi + 1} \} \subseteq \mathbb R$ are algebraically independent on $\mathbb Q$, but the set $\{ \pi, \sqrt{2\pi + 1} \}$ is not. Indeed, $\left( \sqrt{2\pi + 1} \right)^2 - 2\pi -1 = 0$, so the non-zero polynomial $f(t_1,t_2) = t_1^2 - 2t_2 -1$ is null in $\left( \sqrt{2\pi + 1}, \pi \right)$.
\end{example}

To each boundary segments $b$, we associate an indeterminate $y_b$ such that $\y = \left\{ y_b | b \in \BSM \right\}$ is algebraically independent. The elements of $\y$ are called the \emph{coefficients}.
The \df{ground ring} is set as
\[ \mathbb{Z}\mathbb{P} = \mathbb{Z}[y_b^{\pm 1}\mid b\in \mathbf{B}(\mathbf{S}, \mathbf{M})].\]

\begin{definition}
 An \emph{initial seed} is a triplet $(\x,T)$ such that: \begin{itemize}
     \item $T$ is a triangulation of $\SM$;
     \item $\x = \{ x_t | t \in T \}$ is a \emph{cluster} and together with $\y$, it forms a generating set of $\F$.
 \end{itemize}
\end{definition}



In other words, a seed is triangulation with a variable associated to each arc. 

From a given seed, we can obtain a new one by a process called \df{mutation}. First, let's define the mutation only on triangulations. 

\begin{proposition}\cite[Proposition 9]{DP15}
 Let $T$ be a triangulation 
 of a marked surface $\SM$ and let $t$ be an arc of $T$. There exists a unique arc $t' \neq t$ such that $\left(T\setminus \{ t\}\right) \cup t'$ is still a triangulation of $\SM$.
\end{proposition}

\begin{example}
Figure \ref{Fig::M3} depicts the mutation in a M\"obius strip with 3 marked points $M_3$, where each line indicates a mutation. 
\begin{figure}[htbp]\centering\includegraphics[width=1\textwidth]{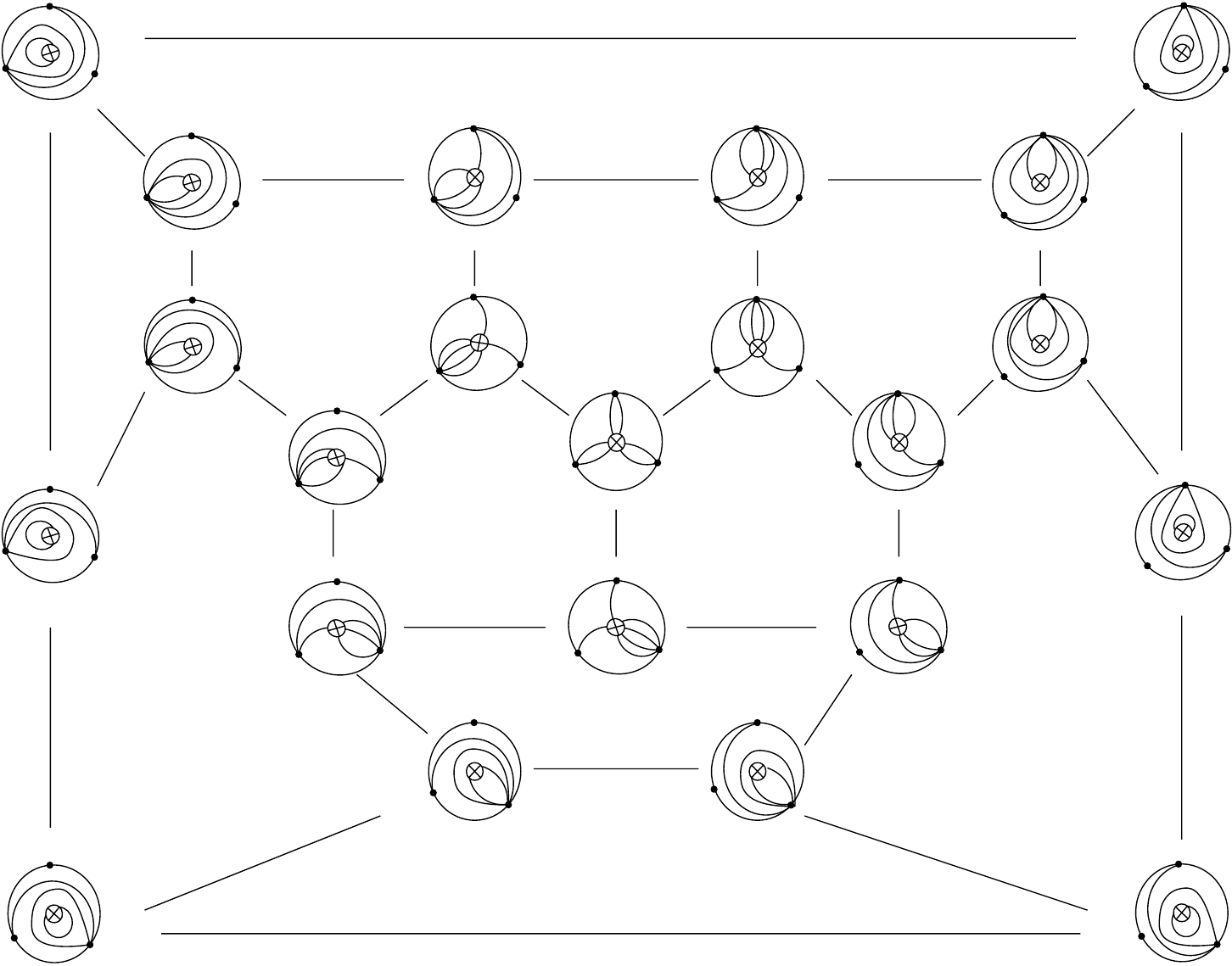}\caption{Mutations on the M\"obius strip with three marked points} \label{Fig::M3}\end{figure}
\end{example}

\begin{definition}
 Let $(T, \x)$ be an initial seed and $t \in T$. The \emph{mutation} of $(T, \x)$ in the direction $t$ transforms $(T, \x)$ in a \emph{seed} $\mu_t(T, \x) = (\mu_t(T), \mu_t(\x))$ where $\x = \x \setminus \{ x_t \} \cup \{ x_{t'} \}$ with $x_{t'}$ such that:
 \begin{enumerate}
    \item If the arc $t$ is the diagonal in a rectangle $abcd$ that creating two triangles as shown in Figure \ref{Fig::Mutation1},
    then the relation is the following: 
    \begin{equation} \label{Eq::Mutation1}
     x_tx_t'=x_ax_c+x_bx_d.
    \end{equation}
    This is called the Ptolemy relation because it is analogous to the relation that holds among the lengths of the sides and diagonals of a cyclic quadrilateral by Ptolemy's theorem.
    \begin{figure}[htbp]\centering\includegraphics[width=1\textwidth]{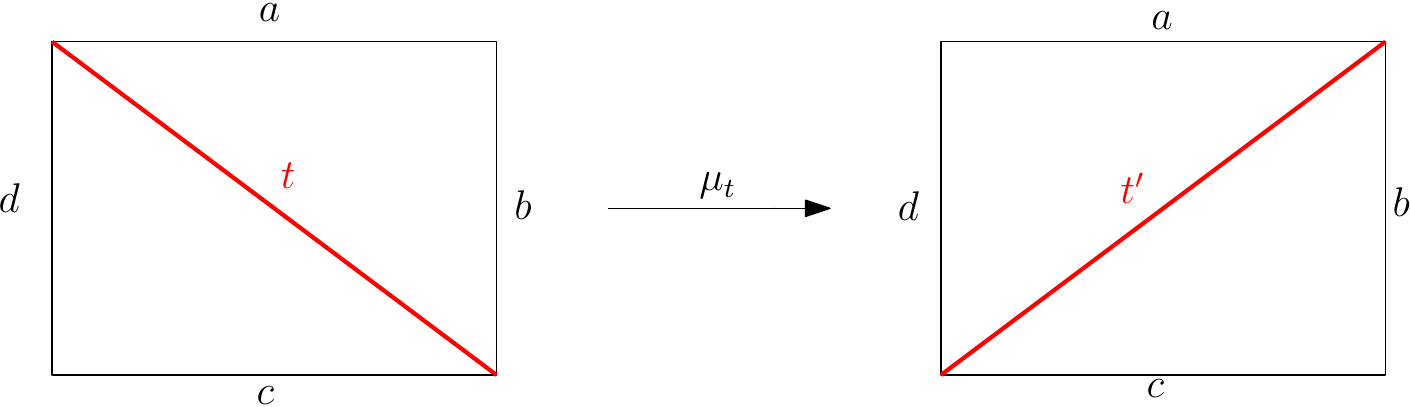}\caption{Mutation of type 1} \label{Fig::Mutation1}\end{figure}
    
    \item 
    Consider an anti-self-folded-triangle, as in Figure \ref{Fig::Mutation2}.
    Let $t$ be the arc corresponding to the two identified edges. Then, the relation is 

    \begin{equation} \label{Eq::Mutation2}
   x_tx_t' = x_a.
    \end{equation}

    \begin{figure}[htbp]\centering\includegraphics[width=0.8\textwidth]{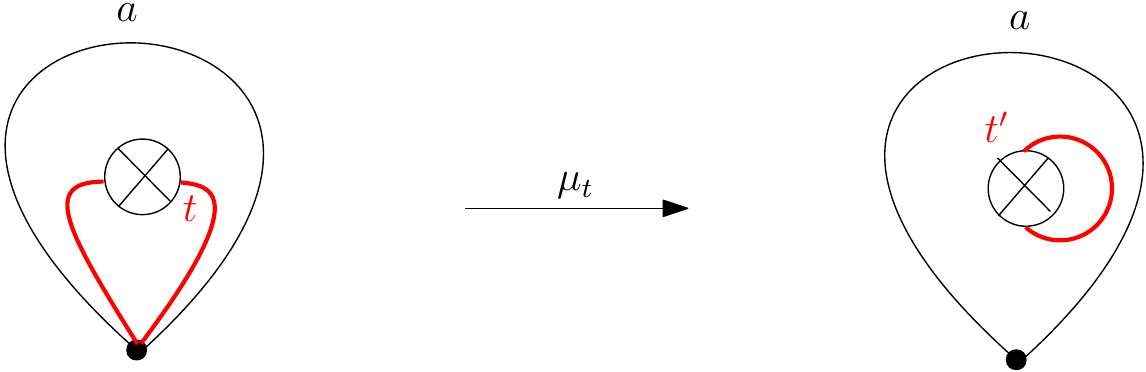}\caption{Mutation of type 2} \label{Fig::Mutation2}\end{figure}
    
    \item If a one-sided closed curve $t$ is in the triangulation with boundary $y$ as in Figure \ref{Fig::Mutation3}, the relation is
    \begin{equation} \label{Eq::Mutation3}
    x_tx_t' = x_a.
    \end{equation}
    
    \begin{figure}[htbp]\centering\includegraphics[width=0.8\textwidth]{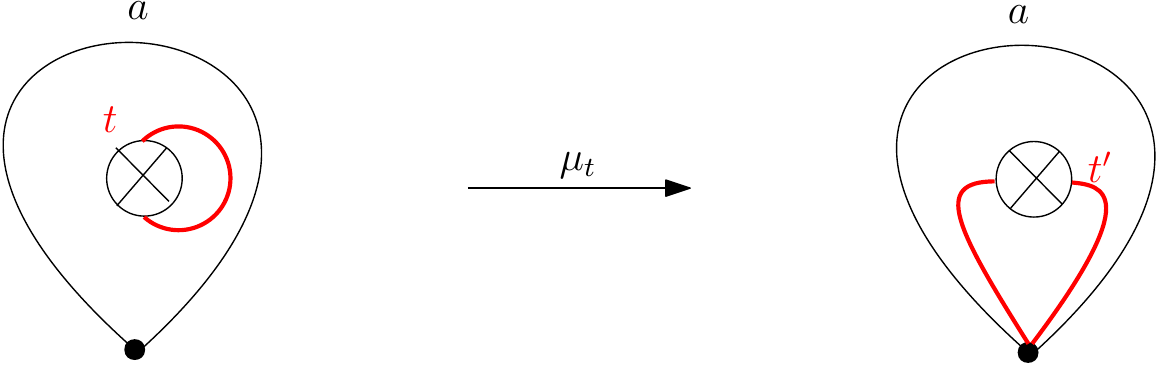}\caption{Mutation of type 3} \label{Fig::Mutation3}\end{figure}
    
    \item Let $t$ be the diagonal in a rectangle with sides $a, b, c, d$, as Figure \ref{Fig::Mutation4} shows below, then the relation is 
    
    \begin{equation} \label{Eq::Mutation4}
    x_tx_t' = (x_b+x_c)^2+x_ax_bx_c.
    \end{equation}
    
    \begin{figure}[htbp]\centering\includegraphics[width=0.8\textwidth]{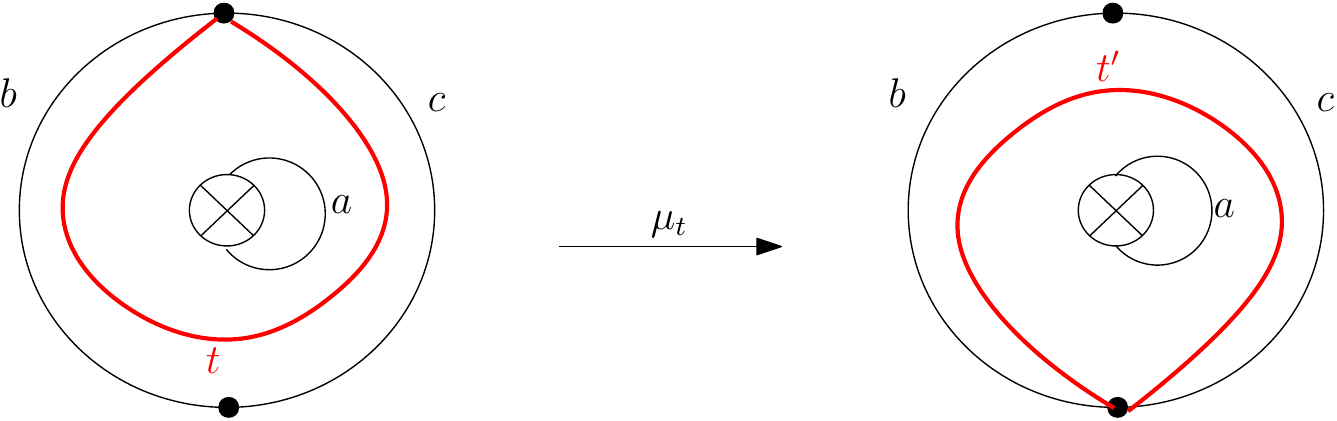}\caption{Mutation of type 4} \label{Fig::Mutation4}\end{figure}
    
 \end{enumerate}
 In general, a \emph{seed} is obtained by a sequence of mutations from the initial seed.
\end{definition}


\begin{example}
Figure \ref{Fig::1mutM3} shows a mutation of a seed in a M\"obius strip with 3 marked points $M_3$. $x_3'$ is the new cluster variable.
\begin{figure}[htbp]\centering\includegraphics[width=0.8\textwidth]{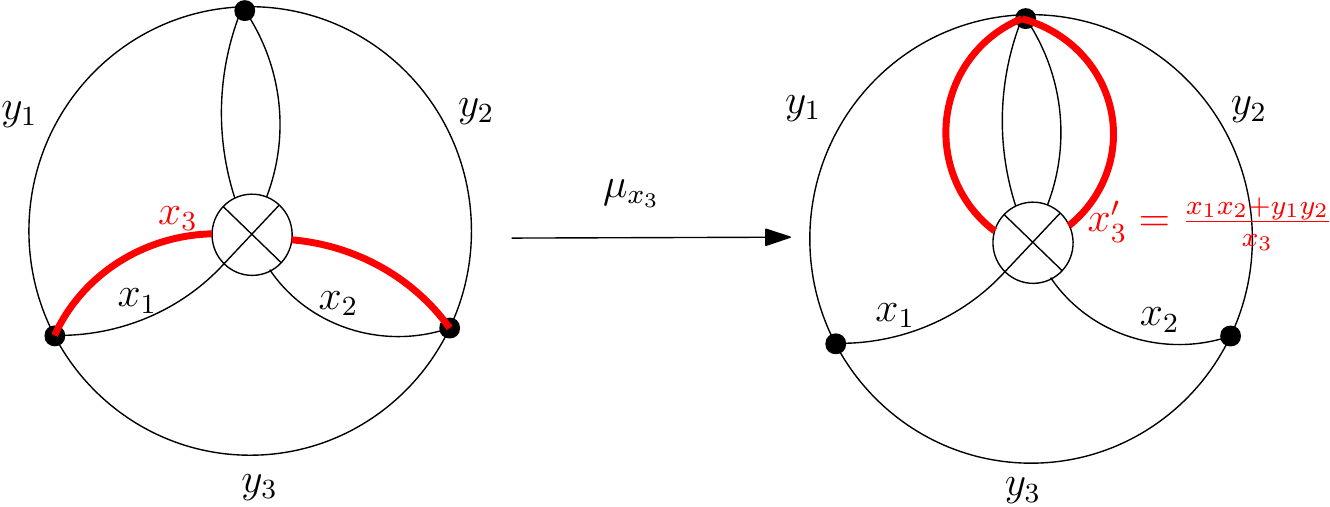}\caption{Mutation of clusters in $M_3$} \label{Fig::1mutM3}\end{figure}
\end{example}

\begin{definition}
Let $\X$ be the union of all clusters obtained by successive mutation from an initial seed $(T, \x)$. The \df{quasi-cluster algebra} is 
\[\mathscr{A}(T, \x) = \ZP\left[ \X \right]. \]
That is to say, every element of $\mathcal F$ which can be expressed as a polynomial in the elements of $\X$ with coefficients in $\ZP$.
\end{definition}


\begin{example}
In figure \ref{Fig::mutM2}, cluster variables from the quasi-cluster algebra from a M\"obius Strip with 2 marked points ($M_2$) are calculated. In the figure, $x_1'$ and $x_2'$ are derived from the Ptolemy relation \ref{Eq::Mutation1}: $$x_1' = \frac{{x_2}^2+y_1y_2}{x_1}, x_2' = \frac{x_1'(y_1+y_2)}{x_2} = \frac{x_1({y_1}^2y_2+y_1{y_2}^2+{x_2}^2y_1+{x_2}^2y_2)}{x_2}.$$ $x_1''$ is mutated from an anti-self-folded triangle according to equation \ref{Eq::Mutation3}, and can be calculated as
$$x_1'' = \frac{x_2'}{x_1'} = \frac{{x_1}^2(y_1+y_2)}{x_2}.$$ $x_2''$ is calculated according to formula \ref{Eq::Mutation4}, which yields $$x_2'' = \frac{(y_1+y_2)^2+{x_1''}^2y_1y_2}{x_2'} = \frac{(y_1+y_2)({x_1}^4y_1y_2+{x_2}^2)}{x_1x_2(y_1y_2+{x_2}^2)}.$$

The cluster algebra consists of all polynomials in $x_1, x_2, x_1', x_2', x_1''$ and $x_2''$ with coefficients in $\Z[y_1, y_1^{-1},y_2, y_2^{-1},y_3, y_3^{-1}]$.


\begin{figure}[htbp]\centering\includegraphics[width=1\textwidth]{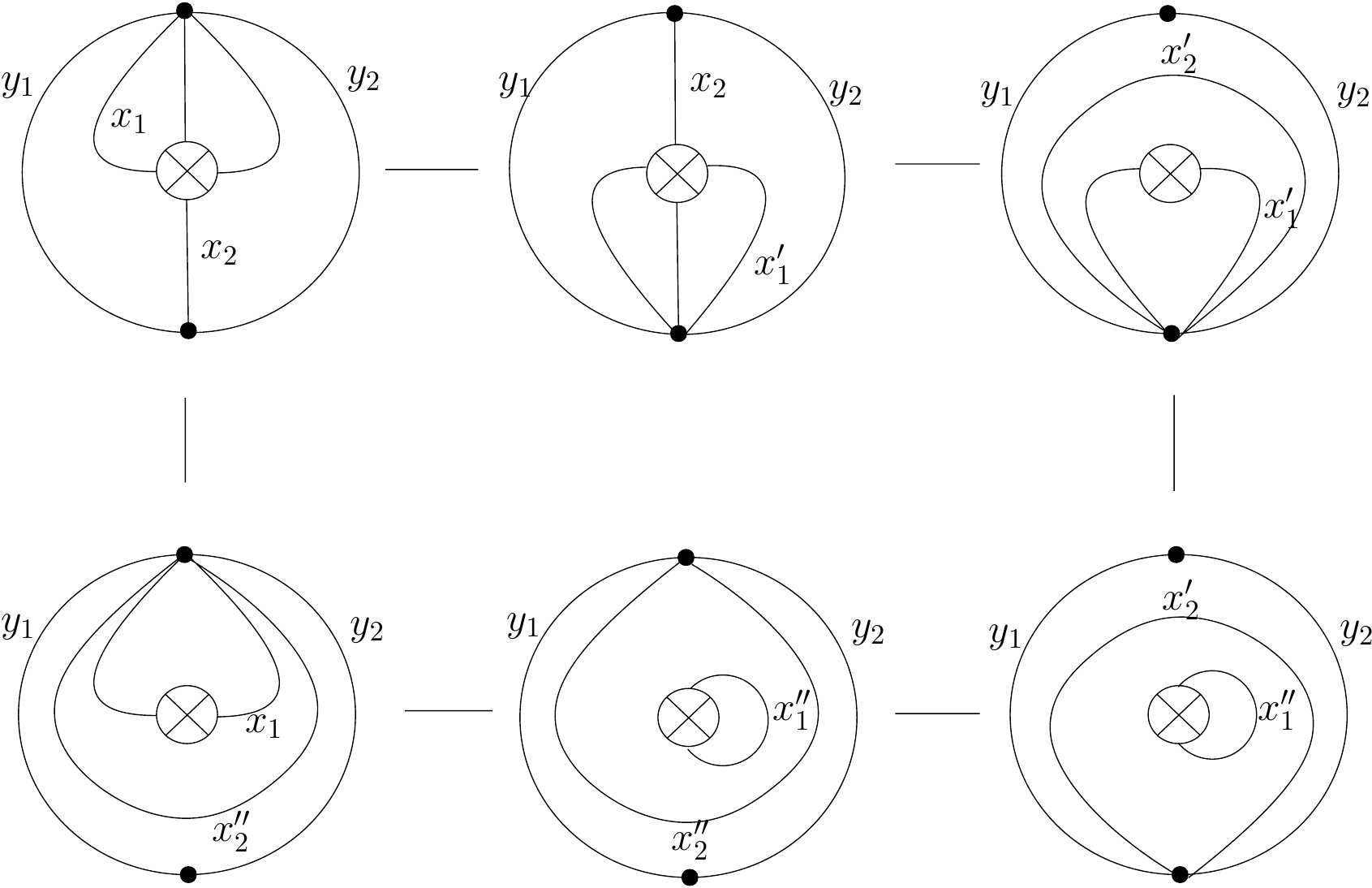}\caption{Cluster variables on $M_2$} \label{Fig::mutM2}\end{figure}

\end{example}

\begin{proposition} \label{Prp::BijClustTriang}
 There is a bijection between the set of triangulations of a marked surface and the set of seeds of the quasi-cluster algebra arising from this surface.
\end{proposition}

\begin{proof}
 Clearly, to any seed $(T, \x)$, we can associate the triangulation $T$. 
 
 Moreover, to any arc, we can associate a unique cluster variable, \cite{DP15}. Wilson also gave a formula to express the cluster variable associated to a given arc in term of the cluster variables in any cluster, \cite{Wil19}. Therefore, to any triangulation $T$, we can associate a unique cluster $(T, \x)$.
 
\end{proof}

\begin{crl}
 There are $4^n + \binom{2n-2}{n-1}$ seeds in a quasi-cluster algebra arising from a M\"obius strip.
 Any other quasi-cluster algebra arising from a non-orientable surface has infinitely many seeds.
\end{crl}

\begin{proof}
 It follows directly from Proposition \ref{Prp::BijClustTriang} and Theorem \ref{Thm::NumberArcs}.
\end{proof}

\section*{Acknowledgements}

We would like to thank Jon Wilson for fruitful discussions. We would also like to thank Hugh Thomas for providing editorial advice of this paper and helping us. 
Finally, we would like to thank Canada/USA Mathcamp for giving us the opportunity to meet together and spark discussions on this topic. 

\bibliographystyle{alpha}
\bibliography{biblio}

\begin{thebibliography}{AHBHY18}

\bibitem[AHBHY18]{AHBHY18}
Nima Arkani-Hamed, Yuntao Bai, Song He, and Gongwang Yan.
\newblock Scattering forms and the positive geometry of kinematics, color and
  the worldsheet.
\newblock {\em Journal of High Energy Physics}, 2018(5), May 2018.

\bibitem[BMRT07]{BMRT07}
Aslak~B. Buan, Robert~J. Marsh, Idun Reiten, and Gordana Todorov.
\newblock Clusters and seeds in acyclic cluster algebras.
\newblock {\em Proceedings of the American Mathematical Society},
  135(10):3049--3060, 2007.

\bibitem[DP15]{DP15}
Grégoire Dupont and Frédéric Palesi.
\newblock Quasi-cluster algebras from non-orientable surfaces.
\newblock {\em Journal of Algebraic Combinatorics}, 42(2):429–472, Mar 2015.

\bibitem[FST08]{FST08}
Sergey Fomin, Michael Shapiro, and Dylan Thurston.
\newblock Cluster algebras and triangulated surfaces. part {I}: Cluster
  complexes.
\newblock {\em Acta Mathematica}, 201(1):83--146, 2008.

\bibitem[FZ02]{FZ02}
Sergey Fomin and Andrei Zelevinsky.
\newblock Cluster algebras {I}: Foundations.
\newblock {\em Journal of the American Mathematical Society}, 15(2):497--529,
  2002.

\bibitem[FZ03]{FZ03}
Sergey Fomin and Andrei Zelevinsky.
\newblock Cluster algebras {II}: Finite type classification.
\newblock {\em Inventiones mathematicae}, 154(1):63--121, 2003.

\bibitem[FZ07]{FZ07}
Sergey Fomin and Andrei Zelevinsky.
\newblock Cluster algebras {IV}: Coefficients.
\newblock {\em Compositio Mathematica}, 143(1):112--164, 2007.

\bibitem[GLS13]{GLS13}
Ch. Geiss, B.~Leclerc, and J.~Schr\"oer.
\newblock Cluster algebras in algebraic {L}ie theory.
\newblock {\em Transform. Groups}, 18(1):149--178, 2013.

\bibitem[GSV10]{GSV10}
Michael Gekhtman, Michael Shapiro, and Alek Vainshtein.
\newblock {\em Cluster algebras and {P}oisson geometry}, volume 167 of {\em
  Mathematical Surveys and Monographs}.
\newblock American Mathematical Society, Providence, RI, 2010.

\bibitem[Lam38]{Lam38}
M.~Lamé.
\newblock Extrait d’une lettre de {M}. {L}amé à {M}. {L}iouville sur cette
  question: Un polygone convexe étant donné, de combien de manières peuton
  le partager en triangles au moyen de diagonales?
\newblock {\em Journal des mathématiques pures et appliquées}, 3:505--507,
  1938.

\bibitem[Wil18]{Wil18}
Jon Wilson.
\newblock Shellability and sphericity of finite quasi-arc complexes.
\newblock {\em Discrete Comput Geom}, 59(3):680--706, 2018.

\bibitem[Wil19]{Wil19}
Jon Wilson.
\newblock Positivity for quasi-cluster algebras, 2019.
\newblock arXiv:1912.12789.

\end{thebibliography}

\end{document}